\documentclass[a4paper,11pt,twoside]{article}

\usepackage[latin1]{inputenc} 
\usepackage[T1]{fontenc} 
\newlength{\minipagewidth}
\usepackage{a4wide}
\usepackage{amsmath,amsthm}
\usepackage{amssymb}
\usepackage{graphicx}
\usepackage{color}
\usepackage{epic}
\usepackage{color}
\usepackage{enumerate}

\newcommand{\pN}[1]{{#1}^N}
\newcommand{\opt}[1]{\widetilde{#1}}
\newcommand{\cc}{{\cal C}}
\newcommand{\al}{\mathcal{A}}

\newcommand{\E}{\mathbb{E}}

\newcommand{\R}{\mathbb{R}}

\renewcommand{\S}{\mathbb{S}}

\newcommand{\calL}{\mathcal{L}}
\newcommand{\calP}{\mathcal{P}}
\newcommand{\calM}{\mathcal{M}}

\newcommand{\calN}{\mathcal{N}}
\newcommand{\calR}{\mathcal{R}}
\newcommand{\calW}{\mathcal{W}}

\newcommand{\eps}{\varepsilon}
\newcommand{\ph}{\varphi}

\newcommand{\Span}{\mathrm{Span}}
\newcommand{\Tr}{\mathrm{Tr}}
\newcommand{\Id}{\mathrm{Id}}
\newcommand{\Law}{{\rm Law}}

\newcommand{\Bigo}{\mathop{{\rm O}}}

\newcommand{\limop}[1]{\mathop{ {\rm #1} }\limits}

\renewcommand{\d}{{\rm d}}
\newcommand{\as}{{\rm a.s.}}

\newcommand{\eqdef}{ \mathop{=}^{{\rm def}} }
\newcommand{\one}[1]{ {\rm l} \hspace{-.7 mm} {\rm l}_{ #1}   }
\newcommand{\dps}{\displaystyle}
\newcommand{\fracd}[2]{\frac{\dps #1}{\dps #2}}
\newcommand{\cond}[1]{ \, #1 \vert \, }

\newcommand{\abs}[1]{\left | #1\right |}
\newcommand{\set}[1]{\left\{#1\right\}}
\newcommand{\pare}[1]{ \left(#1\right) }
\newcommand{\normop}[1]{\left \vert \! \left  \vert \! \left \vert  #1 \right \vert \! \right  \vert \! \right  \vert }

\newcommand{\bracket}[1]{\left \langle #1\right \rangle}

\newcommand{\barr}[1]{\left. \begin{array}{#1}}
\newcommand{\earr}{\end{array}\right.}
\newcommand{\bmat}{\begin{pmatrix}}
\newcommand{\emat}{\end{pmatrix}}
 
\theoremstyle{plain}
\newtheorem{The}{Theorem}[section]
\newtheorem{Lem}[The]{Lemma}
\newtheorem{Pro}[The]{Proposition}

\newtheorem{Cor}[The]{Corollary}
\newtheorem{Def}[The]{Definition}

\theoremstyle{definition}
\newtheorem{Rem}[The]{Remark}

\title{Scalable and Quasi-Contractive \\ Markov Coupling of Maxwell Collisions}

\author{Mathias Rousset\footnote{INRIA Paris-Rocquencourt, Domaine de Voluceau - Rocquencourt, B.P. 105 - 78153 Le Chesnay}\,\,\,\footnote{Université Paris-Est, CERMICS (ENPC),  6-8 Avenue Blaise Pascal, Cité Descartes ,  F-77455 Marne-la-Vallée}
}
\begin{document}

\maketitle

\begin{abstract}
This paper considers space homogenous Boltzmann kinetic equations in dimension $d$ with Maxwell collisions (and without Grad's cut-off). An explicit Markov coupling of the associated conservative (Nanbu) stochastic $N$-particle system is constructed, using plain parallel coupling of isotropic random walks on the sphere of two-body collisional directions. The resulting coupling is almost surely decreasing, and the $L_2$-coupling creation is computed explicitly. Some quasi-contractive and uniform in $N$ coupling / coupling creation inequalities are then proved, relying on $2+\alpha$-moments ($\alpha >0$) of velocity distributions; upon $N$-uniform propagation of moments of the particle system, it yields a $N$-scalable $\alpha$-power law trend to equilibrium. The latter are based on an original sharp inequality, which bounds from above the coupling distance of two centered and normalized random variables $(U,V)$ in $\R^d$, with the average square parallelogram area spanned by $(U-U_\ast,V-V_\ast)$, $(U_\ast,V_\ast)$ denoting an independent copy. Two counter-examples proving the necessity of the dependance on $>2$-moments and the impossibility of strict contractivity are provided. The paper, (mostly) self-contained, does not require any propagation of chaos property and uses only elementary tools. 
\end{abstract}

\tableofcontents

\section{Summary and contents}
\subsection{Summary}
This paper considers space homgenous Boltzmann kinetic equations in dimension $d$, with conservative (two-body) collisions of Maxwell type. Emphasis is made on the associated conservative (or Nanbu) $N$-particle system; a Markov process denoted (up to particle permutations ${\rm Sym}_N$)
\begin{equation}
  \label{eq:part_sys_nocoupled}
  t \mapsto \pN{V}_t = (\pN{V}_{t,(1)}, \hdots,\pN{V}_{t,(N)}) \in \pare{ \R^{d} }^N \quad [ {\rm Sym}_N],
\end{equation}
and satisfying the following conservation laws (for any $t \geq 0$):
\begin{equation}
  \label{eq:cons}
  \dps \bracket{\pN{V}_t}_{N} = 0 \quad {\as},  \qquad 
  \dps \bracket{\abs{\pN{V}_t}^2}_{N} =1 \quad  {\as}, 
\end{equation}
where in the above, the bracket denotes the averaging over particles ($\bracket{\, . \, }_{N} \equiv \frac{1}{N} \sum_{n=1}^N$). Moreover, the latter process is reversible with respect to the invariant uniform probability distribution
\[
{\rm Unif}^N_{0,1} \eqdef {\rm Unif} \set{v \in \pare{ \R^{d} }^N \cond{\Big} \text{\eqref{eq:cons} holds}},
\]
equivalently defined as the Riemannian volume of the associated sphere, or as the conditional distribution associated with the average momentum and energy observables.

The Markov dynamics of~\eqref{eq:part_sys_nocoupled} is specified by two-body random collisions of (Levy) jump type, satisfying \emph{two-body} conservation of momentum and kinetic energy (particles have identical masses). By definition of Maxwell collisions, the rate of the two-body collisions is constant; and by Galilean invariance, a random two-body collision is necessarily an isotropic random step on the euclidean sphere of possible collisional directions. The \emph{collisional direction} is the direction of the velocity difference of a particle pair. The angular size ($\theta \in [0,\pi]$) of the step is the so-called \emph{scattering (or deviation)} angle.The \emph{angular collision kernel} $b(\d \theta)$ is a positive Levy measure on $[0,\pi]$ generating the random steps of the scattering angle, with Levy condition
\begin{equation}
  \label{eq:Levy}
  \lambda \eqdef \int_{[0,\pi]} \sin ^2 \theta  \, b( \d \theta) < +\infty.
\end{equation}
More notation on the collision processes of interest are given in Section~\ref{sec:coupl}, and standard details, especially intended for the unfamiliar audience, are given in Section~\ref{sec:notation}.

It is thus possible to construct an explicit Markov coupling (\textit{i.e.} a probabilistic coupling of two copies of a Markov process which is itself again Markov) denoted
\begin{equation}
  \label{eq:part_sys}
  t \mapsto (\pN{U}_t,\pN{V}_t)\equiv (U_{t,(1)},V_{t,(1)}, \hdots,U_{t,(N)},V_{t,(N)}) \in \pare{ \R^{d} \times \R^{d} }^N \quad [ {\rm Sym}_N ],
\end{equation}
which is \emph{globally} invariant by particle permutation, such that $\pN{U}_t \in \pare{ \R^{d} }^N$ and $\pN{V}_t \in \pare{ \R^{d} }^N$ indpendently satisfy the conservation laws~\eqref{eq:cons}, and such that collisions are coupled using the following set of rules:
\begin{enumerate}[(i)]
\item Collision times and collisional particles are the same.
\item  Scattering angles are the same.
\item The isotropic random step on the collisional direction is coupled using plain parallel transport (in the geometric sense), with no reflexion.
\end{enumerate}
The sphere being a strictly positively curved manifold, the latter coupling is bound to be almost surely decreasing, in the sense that for any initial condition and $0 \leq t \leq t+h$
\begin{equation}
  \label{eq:contr_as}
  \bracket{ \abs{\pN{U}_{t+h}-\pN{V}_{t+h}}^2}_N \leq \bracket{ \abs{\pN{U}_{t}-\pN{V}_{t}}^2}_N \quad \as.
\end{equation}

The goal of this paper is to study the quasi-contractivity of the latter coupling, uniformly in the number of particles $N$. We first compute the coupling creation by computing the time derivative of the average coupling distance
\begin{equation}\label{eq:contr}
  \fracd{\d}{\d t} \E  \bracket{ \abs{\pN{U}_{t}-\pN{V}_{t}}^2}_N = - \E  \bracket{  \cc \pare{ \pN{U}_t , \pN{V}_t , \pN{U}_{\ast,t} , \pN{V}_{\ast,t}  } }_N \leq 0.
\end{equation}
In the above, and in the rest of the paper, the following notation is used
\[
\bracket{o\pare{\pN{u},\pN{v},\pN{u}_{\ast },\pN{v}_{\ast}} }_{N} \eqdef \fracd{1}{N^2} \sum_{n_1 , n_2 =1}^{N} o(\pN{u}_{(n_1)},\pN{u}_{(n_1)},\pN{u}_{(n_2)},\pN{v}_{(n_2)}),
\]
in order to account for averages over particles of a two-body observable $o: \pare{\R^d\times \R^d}^2 \to \R$. It turns out that the two-body ``coupling creation'' functional in the right hand side of~\eqref{eq:contr} is given by the following alignement between the velocity difference $v-v_\ast \in \R^d$, and its coupled counterpart $u-u_\ast \in \R^d$:
\begin{align}\label{eq:coupl_crea}
\cc(u,v,u_\ast,v_\ast) =  \lambda \fracd{c_{d-1}}{c_{d-3}}  \abs{u-u_\ast}\abs{v-v_\ast} - (u-u_\ast ) \cdot (v-v_\ast)  \geq 0.
\end{align}
In the above $\dps {c}_d = \int_{0}^{\pi/2} \sin^{d}(\ph) \, \d \ph$ denotes the $d$'th Wallis integral, and $\lambda >0$ is the variance of the scattering angle kernel, as defined in~\eqref{eq:Levy}. 

In order to relate the coupling and the coupling creation, we will introduce in the present paper an original general sharp inequaIity (see~\eqref{eq:fund_ineq} below) that holds for any couple of centered and normalized random variables in $\R^d$. In the present context, it takes the following form
\begin{align}
 f\pare{ \bracket{\abs{\pN{u}-\pN{v}}^2}_N } \leq&  \min\pare{\kappa_{\bracket{\pN{u}\otimes \pN{u}}_N},\kappa_{\bracket{\pN{v}\otimes \pN{v}}_N} } \nonumber \\
&\times \bracket{ \abs{\pN{u}-\pN{u}_\ast}^2\abs{\pN{v}-\pN{v}_\ast}^2 - \pare{\pare{\pN{u}-\pN{u}_\ast} \cdot \pare{\pN{v}-\pN{v}_\ast}}^2}_N , \label{eq:fund_ineq_uv}
  \end{align}
for any vectors $\pN{u} \in (\R^d)^N$ and $\pN{v}  \in (\R^d)^N$ both satisfying the conservation laws~\eqref{eq:cons}. In the above, the condition number
\begin{equation}
  \label{eq:kappa}
   \kappa_{S} \eqdef \pare{1- \normop{S} }^{-1} \in [\frac{d}{d-1},+\infty]
\end{equation}
is a function of the spectral radius $\normop{S} \leq 1$ of a positive trace $1$ symmetric matrix; and
\begin{equation}
  \label{eq:f}
  \begin{array}{ccccc}
    f & : & [0,4] & \to & [0,1] \\
    & & x & \mapsto & x-\fracd{x^2}{4}, \\
  \end{array}
\end{equation}
is a positive concave function (used througout the paper) satisfying $f(x) \limop{\sim}_{x \to 0} x$, as well as $f(4-x)=f(x)$ which ensures the symmetry $\pN{v} \to -\pN{v}$ in~\eqref{eq:fund_ineq_uv}. Moreover, the equality case in~\eqref{eq:fund_ineq_uv} is satisfied under the following sufficient two conditions:
\begin{enumerate}
\item Co-linearity of $\frac{u_{(n)}}{\abs{u_{(n)}}}$ and $\frac{v_{(n)}}{\abs{v_{(n)}}}$ for all $1 \leq n \leq N$; 
\item Isotropy of co-variances $\bracket{\pN{u}\otimes \pN{u}}_N=\bracket{\pN{u}\otimes \pN{v}}_N=\frac{1}{d} \Id$ or $\bracket{\pN{v}\otimes \pN{v}}_N=\bracket{\pN{u}\otimes \pN{v}}_N =\frac{1}{d} \Id$.
\end{enumerate}

It is then of interest to compare that the alignement functional in the right hand side of~\eqref{eq:fund_ineq_uv} (which is a sharp upper bound of the square coupling distance), and the coupling creation functional~\eqref{eq:coupl_crea}. They differ by a weight of the form $\abs{u-u_\ast}\abs{v-v_\ast}$ which forbids any strong ``coupling/coupling creation'' inequality of the form
\[
\fracd{f \pare{ \bracket{\abs{\pN{u}-\pN{v}}^2}_N } }{ \bracket{\cc(\pN{u},\pN{v},\pN{u}_\ast,\pN{u}_\ast)}_N } \leq r < +\infty 
\]
for some universal constant $r >0$ independant of $N$ and of the the pair $(\pN{u},\pN{v})\in (\R^d \times \R^d)^N$ both satisfying the conservation laws~\eqref{eq:cons}. However, a direct Hölder inequality yields some weaker power law versions (see Section~\ref{sec:results}), for any $\alpha \in ]0,+\infty[$:
\begin{equation}
  \label{eq:rate}
  \fracd{ f \pare{ \bracket{ \abs{\pN{u}-\pN{v}}^2}_N  } }{ \bracket{  \cc \pare{ \pN{u} , \pN{v} , \pN{u}_{\ast} , \pN{v}_{\ast}  } }_N ^{\frac{\alpha}{1+\alpha}} } \leq r_{\alpha,\pN{u},\pN{v}} < +\infty ,
\end{equation} 
 The inequality~\eqref{eq:rate} can be interpreted as a power law (of order $\alpha$) estimate of the coupling (quasi-)contractivity (\textit{i.e.} with speed $\limop{\sim}_{t \to + \infty} t^{-\alpha}$); and as expected, $r_{\alpha,\pN{u},\pN{v}}$ can be controlled by any \emph{$N$-averaged finite moments of order $>2 + \alpha $} of the velocity distributions. 

Such results are similar to the classical results (\cite{Cer82,CarCar92,BobCer99,CarGabTos99,TosVil99_trend}) that are obtained with entropy methods using ``entropy/entropy creation'' inequalities. Yet, the coupling method has some noticeable specificities:
\begin{description}
\item[Maxwell restriction] The analysis is restricted to Maxwell collisions. 
\item[Angular condition] The analysis is independent of the scattering angular distribution of collisions.
\item[Particle system size] The analysis is independent of the particle system size $N$. It works similarly for the kinetic equation.
\item[A priori estimates] The analysis depends on higher $>2$ moments of velocity distributions, and not on regularity estimates. Such estimates are available for the kinetic equation, but unfortunately, up to our knowledge, not directly for the $N$-particle system in the spirit of~\cite{IkeTru56} . However, they can be obtained (indirectly and at least in principle) by pull-back, by using uniform in time large $N$ propagation of chaos, as proven in~\cite{MisMou11}.
\item[Constants] Constants are simple and explicit.
\item[Sharpness] The weak ``coupling/coupling creation inequalities'' are derived via Hölder inequality from the key sharp inequality~\eqref{eq:fund_ineq_uv}.
\end{description}
 The latter method then yields some lower bound estimates on the contraction rate (of the particle system distribution, or of the kinetic equation) with respect to the $L^2$-Wasserstein metric. In the case of the kinetic equation with Maxwell collisions, the latter contraction has been shown to be strictly positive (thus implying uniqueness) in the classical paper by Tanaka \cite{Tan78}. In a sense, the analysis in the present paper makes Tanaka's argument quantitative.

We also suggest some negative results in the form of two counterexamples to stronger versions of ``coupling/coupling creation inequalities''. These are the counterpart of counterexamples to Cercignani's conjecture (\cite{Bob88,BobCer99,VilCer03}) in the entropy context.
\begin{enumerate}[(i)]
\item Velocity distributions with sufficiently heavy tails can make the coupling creation vanish. This first counterexample shows that any ``coupling/coupling contraction inequality'' must involve some higher order (say, $>2$) velocity distribution moments.
\item There exists a continuous perturbation of the identity coupling at equilibrium for which however the coupling creation is sub-linearly smaller than the coupling itself. This second type of counterexample shows that even with some reasonable moment or coupling domain restrictions, a sub-exponential trend is unavoidable.
\end{enumerate}

\subsection{Contents}
In Section~\ref{sec:context}, we summarize the literature related to the present work. In Section~\ref{sec:systems}, we recall some notation and basic concepts related to probabilistic couplings for Markov particle systems. In Section~\ref{sec:coupl}, we define the random collisions coupling of interest. In Section~\ref{sec:kinetic}, we detail the associated kinetic equations. In Section~\ref{sec:results}, the main results of the present paper are stated. In Section~{\ref{sec:cex}}, the two counter-examples showing the necessity of a sub-exponential trend and of finite higher moments are presented and proved.

In Section~\ref{sec:notation} and sub-sections there in, standard facts on conservative random collisions and associated particle systems are recalled.

In Section~\ref{sec:coupl_2}, an explicit (\textit{i.e.} coordinate) formulation of spherical couplings is detailed. In Section~\ref{sec:contr}, the contractivity of spherical couplings is computed. In Section~\ref{sec:coupl_nanbu}, the coupling creation of the coupled particle system is, in turn, computed. In Section~\ref{sec:proofs}, the proof of the quasi-contractivity in appropriate Wasserstein distance of the kinetic equation and of its associated conservative particle system is detailed.

In Section~\ref{sec:spec}, the special inequality between coupling distance and colinearity of coupled pairs is proven. In Section~\eqref{sec:holder}, the quasi-contractive coupling / coupling creation estimate is deduced.

\section{Context and results}
\subsection{Context}\label{sec:context}
The mathematical literature on the space homogenous Boltzmann kinetic equations (and related models) is extremely vast, especially in the case of Maxwell collisions, and we refer to the classical reviews~\cite{Cer69,Vil02}. In the same way, the use of explicit coupling methods to study the trend to equilibrium of Markov processes (or Markov chains) is now a classical topic on its own, especially for discrete models (see \textit{e.g.} \cite{LevPerWil09}). It is also a well-established topic for continuous models, as well as for non-linear partial differential equations that have an interpretation in terms of Markovian particles.  Let us mention some classical papers more closely related to the present study, with a sample of more recent references.

\begin{description}
\item[Well-posedness] For Maxwell molecules, well-posedness of the kinetic equation as a probability flow using a probabilistic coupling was initiated in~\cite{Tan78}, with a sophisticated probabilistic technology. In \cite{PulTos96,TosVil99}, an ad hoc spectral theory (Fourier) enabled to simplify the argument, and to push forward the theory. A recent example of the use of the coupling method to non-Maxwell molecules is available in~\cite{FouMou09}.
\item[Moments] Moments production and moments propagation for the Boltzmann kinetic equation is a well established phenomenon. For Maxwell molecules, some explicit calculations are possible (\cite{IkeTru56}), and finite vs. infinite moments strictly propagates. In general, higher moments estimates typically rely on the so-called Pozner inequality (\cite{Wen97,Bob97,CarGabTos99}). It is still a field of study (\cite{LuMo12,AloCanGamMou13}), with emphasis on moments production without any regularity theory. However, we are not aware of such results directly carried out on the associated $N$-particle systems, which would be very complementary to the present results.
\item[Trend to equilibrium (i)] As said before, most of the results on trend to equilibrium for Boltzmann kinetic equations, rely on ``entropy / entropy creation'' analysis (\cite{Cer82,CarCar92,BobCer99,CarGabTos99,TosVil99_trend}), which try to circumvent the breakdown of Cercignani's conjecture (counterexamples: \cite{Bob88,BobCer99,VilCer03}). It is of interest to interpret the latter counterexamples with the Markovian viewpoint: they amount ot  the impossibility of obtaining a uniform in $N$ log-Sobolev (in the sense of entropy/entropy creation) inequality of the assoicated $N$-particle systems. For Maxwell collisions, or Kac's caricature, special simplifications enables to push forward the trend estimations, for instance using Wild's expansion (\cite{CarLu03}), or central limit theorem to obtain sharp exponential rates (\cite{DolGabReg09,DolReg10}).
\item[Trend to equilibrium (ii)] For other models of non-linear partial differential equations, a Markov coupling can give exponential trend to equilibrium, by using a ``strong coupling/coupling creation inequalitiy'' (see for instance \cite{Mal01,BolGenGui12,BolGenGui12bis}, for non-linear Fokker-Planck models with convexity assumption on potentials). However, the latter cases differ from the case of Boltzmann collisions by their exponential behavior. Indeed, and informally speaking, with respect to the following classical chain of implications (see~\cite{OttVil00}, using the Otto calculus viewpoint) for classical reversible diffusions on manifolds
  \begin{align*}
    &\text{Contractive Markov coupling} \mathop{\Rightarrow}^{\text{(optimal coupling)}} \text{Wasserstein contractivity} \\ 
    & \qquad \mathop{\Rightarrow}^{\text{(Otto differential calculus)}} \text{ Bakry-Emery type convexity criterion} \\
& \qquad \mathop{\Rightarrow}^\text{(two times time derivation)} \text{ Log-Sobolev inequality },
  \end{align*}
the underlying constants in the latter models are in fact \emph{uniform in $N$}. From the breakdown of Cercignani's conjecture, and assuming as an informal \textit{rationale} that the latter ideas still hold for jump processes, strictly contractive Markov coupling cannot be constructed for Boltzmann collision processes.
\item[Trend to equilibrium (iii)] Direct probabilistic methods studying the trend of equilibrium of the Boltzmann-Maxwel $N$-particle system (or Kac's caricature) have been undertaken \cite{DiaSal00,CarCarLos03,CarCarLos08,Oli09}. The main striking feature of the latter list is the difficulty to achieve the so-called ``Kac's program'' (\cite{MisMou11}),  by obtaining a scalable (in $N$) analysis of the trend to equilibrium of the particle system. In~\cite{CarCarLos03,CarCarLos08}, the ananalysis computes exact spectral gaps, which are not $N$-scalable measures of trends to equilibrium. In~\cite{Oli09}, an explicit coupling method is used but the estimate focus on the contraction constant with optimal scale ($\Bigo(\ln N)$) which still is not $N$-scalable. 
\item[Propagation of chaos] In~\cite{MisMou11} have reversed the latter point of view, and proved the trend to equilibrium of the $N$-particle system by pulling-back the long time stability of the kinetic (mean-field $N=+\infty$ limit) equation with uniform in time propagation of chaos. Moreover, the use of coupling methods in this context have been undertaken in~\cite{FouMis13} to obtain sharper chaos propagation rates, at the price of time stability.  These are nonetheless fairly indirect and impressively technical viewpoints. 
\end{description}

In the present work, the whole analysis itself is independant of $N$, and the study is fully elementary, the only advanced tool being the final Hölder inequality. This generality and simplicity is one of the main motivation of the present work.

\subsection{Markov coupling and particle systems}\label{sec:systems}
Let us now make a brief summary of some basic concepts related to coupling methods for Markov particle systems. Consider an exchangeable (particle permutation symmetric) and coupled $N$-particle system of the form~\eqref{eq:part_sys}. Strictly speaking, the state space is obtained by quotienting \emph{globally} $\pare{ \R^{d} \times \R^{d} }^N \! \!  / {\rm Sym}_N$ with particle permutations (the coupling breaks the permutation symmetry, and the system \emph{is not} in the product space $\pare{ \R^{d} }^N \! \!  / {\rm Sym}_N \times \pare{ \R^{d} }^N \! \!  / {\rm Sym}_N$). 


 We also assume that the latter is a \emph{two-body interaction} system, so that its Markov generator has the usual structure:
\begin{equation}
  \label{eq:full_gen_c}
  \calL_c^N \eqdef \fracd{1}{N} \sum_{n,m =1}^{N} L_c^{(n,m)} \pare{= N \times \bracket{ L_c }_{N}},
\end{equation}
where $L_c$ is a Markov generator with state space $ \pare{ \R^{d} \times \R^{d}}^2$, the superscript $(n,m)$ denotes the action on the corresponding pair of particles, and the particle averaging $\bracket{ \quad  }_{N}$ have been obviously extended to opertaors. 

We say that the latter process is a (time homogenous) \emph{Markov coupling} if the following two conditions hold:
\begin{enumerate}[(i)]
\item The marginal distribution of the two processes $t \mapsto \pN{U}_t \in \pare{\R^{d}}^N $ and $t \mapsto \pN{V}_t \in \pare{ \R^{d} }^N $ are permutation symmetric and Markov with the same generator of the form
\begin{equation}
  \label{eq:full_gen}
  \calL^N \eqdef \fracd{1}{N} \sum_{n,m =1}^{N} L^{(n,m)}\pare{= N \times \bracket{ L }_{N}}.
\end{equation}
\item If $\pN{U}_0=\pN{V}_0 \, \as$, then $\pN{U}_t=\pN{V}_t \, \as$ for any $t \geq 0$.
\end{enumerate}
Formally, the first point is direct consequence of the fact that for any $(u,v) \in \R^d \times \R^d$
\begin{equation}
  \label{eq:coupl}
L_c(\ph \otimes \one{} )(u,v) = L(\ph)(u), \quad  L_c(\one{} \otimes \ph )(u,v) = L(\ph)(v),
\end{equation}
 for any test function $\ph$ ranging in a domain of the generator $L$.

Assuming that the coupling is almost surely increasing (\eqref{eq:contr_as} holds), the two-body interaction structure implies that \emph{$L^2$ coupling creation} as defined by~\eqref{eq:contr} is given by a two-body positive \emph{coupling creation} functional
\[
\cc: \pare{\R^{d} \times \R^{d}}^2 \to \R^+
\]
given by the general formula
\begin{equation}
  \label{eq:cc_func}
\cc(u,v,u_\ast,v_\ast) \eqdef   - L^c\circ \pare{(u,v,u_\ast,v_\ast) \mapsto \abs{u-v}^2 + \abs{u_\ast-v_\ast}^2  }.
\end{equation}

Next, for any coupled random variables $(\pN{U},\pN{V}) \in (\R^d \times \R^d)^N$, and any $\alpha > 0$, we will consider  weak ``coupling / coupling creation'' inequalities of the form~\eqref{eq:rate}. If the particle system has an invariant probability distribution, and if the constant in the right hand side of~\eqref{eq:rate} is uniformly bounded in time and particle size $N$, a computable $N$-uniform trend to equilibrium of power law type $\limop{\sim}_{t \to + \infty} t^{-\alpha}$ can be obtained. If the latter holds for any large $\alpha > 0$, the coupling may be called \emph{$N$-uniform quasi-contractive}.

Weak forms of contractivity in terms of Wasserstein metric can be obtained as a corollary. We can first introduce the ``two-step'' Wasserstein distance on exchangeable (permutation symmetric) particle systems.
\begin{Def}
  Let $\pN{U} \in (\R^d)^N $ and $\pN{V} \in (\R^d)^N$ two exchangeable (with ${\rm Sym}_{N} $-invariant distribution) random vectors satisfying the conservation laws~\eqref{eq:cons}. Denote by
\[
\eta_{U^N} \eqdef \frac{1}{N} \sum_{n=1}^{N} \delta_{U_{(n)}} \in \calP(\R^d)
\]
the associated empirical distributions. Then the two-step $L^2$-Wasserstein distance is defined by
\begin{equation*}
  \label{eq:3}
  d_{\calW^N_2}\pare{\Law(\pN{U}),\Law(\pN{V}) } \eqdef d_{\calW_2,(\calP(\R^d),d_{\calW_2})} \pare{\Law(\eta_{\pN{U}}),\Law(\eta_{\pN{V}})}
\end{equation*}
between $\Law(\pN{U})$ and $\Law(\pN{V})$ is induced by the $L^2$-Wasserstein distance between $\Law(\eta_{\pN{U}})$ and $\Law(\eta_{\pN{V}})$ on the space of random probability distributions in $\calP(\calP(\R^d))$, where $\calP(\R^d)$ is itself endowed with the usual $\R^d$-euclidean $L_2$-Wasserstein distance. It is equivalent to the $L_2$-Wassertsein in the quotient space $(\R^d)^N / {\rm Sym_N}$ endowed with the quotient (orbifold) distance, defined by infimum for $(\pN{u},\pN{v})\in (\R^d)^N / {\rm Sym_N} \times (\R^d)^N / {\rm Sym_N}$:
\[
\inf_{\sigma \in {\rm Sym}_N} \pare{\bracket{ \abs{\pN{u}-\pN{v}_{\sigma(\,.\,)} }^2}_N}. 
\]
\end{Def}
This yields the following general upper bound for any $t_0 \geq 0$
\begin{equation}
  \label{eq:ineq_wass}
  \frac{\d^+}{\d t} d_{\calW^N_2}^{2}\pare{ \Law(\pN{U}_t),\Law(\pN{V}_t) } \big \vert_{t=t_0} \leq - \E\bracket{ \cc \pare{ \pN{\opt{U}}_{t_0}  , \pN{\opt{V}}_{t_0} , \pN{\opt{U}}_{t_0,\ast} , \pN{\opt{V}}_{t_0,\ast}  } }_N.
\end{equation}
 In the above, $(\pN{\opt{U}} ,\pN{\opt{V}}) \in (\R^d \times \R^d)^N$ is a random variable representation of any optimal coupling and $\frac{\d^+}{\d t}$ is the right derivative which exists in $[0,+\infty]$ by monotony.

\subsection{Coupled collisions}\label{sec:coupl}
As usual, the post-collisional velocities of a particle pair are given by the collision mapping ($n'_v \mapsto (v',v'_\ast)={\rm coll}_{v,v_\ast}(n'_v)$)
\begin{equation}
\label{eq:collision}
    \begin{cases}
      v' = \frac{1}{2}(v+v_\ast) + \frac{1}{2} \abs{v-v_\ast} n'_v, \\[2pt]
      v'_\ast = \frac{1}{2}(v+v_\ast) - \frac{1}{2} \abs{v-v_\ast} n'_v,
    \end{cases}
  \end{equation}
where $(n_v,n'_v) = \pare{\frac{v-v_\ast}{\abs{v-v_\ast}},\frac{v'-v'_\ast}{\abs{v'-v'_\ast}}} \in \S^{d-1} \times \S^{d-1}$ denote the collisional/post-collisional directions. A Maxwell collision process for two particles in $\R^d\times \R^d$ is the Markov process contructed from the following Levy generator on $\R^{d}\times \R^d$:
\begin{equation}
   \label{eq:gen_levy}
   L(\ph)(v,v_\ast) \eqdef \int_{\S^{d-1} \times [0,\pi]} \pare{\ph(v',v'_\ast) - \ph(v,v_\ast)}  c_{\theta}(n_v, \d n'_v) \, b(\d \theta),
 \end{equation}
where in the above $\ph$ is a test function, $b$ is an angular collisional kernel satisfying the Levy condition~\eqref{eq:Levy}, and $c_{\theta}$ is the isotropic probability transition on the collisional direction (the sphere $\S^{d-1}$) with prescribed scattering (or deviation) angle $\theta$. 

The $2$-body coupled process is then obtained by (i) coupling the scattering angles $\theta$; (ii) coupling the post-collisional directions $n_u'$ and $n_v'$, using \emph{parallel transport} between $n_u$ and $n_v$ on the sphere $\S^{d-1}$ (no reflexion). We state without proof (the reader may resort to a drawing here) two equivalent elementary descriptions of the parallel transport coupling on the sphere. We will call the latter \emph{spherical coupling}.
\begin{Def}\label{def:coupling}
 Let $(n_u,n_v) \in \pare{ \S^{d-1} }^2$ be given with condition $n_u \neq -n_v$. There is a unique rotation of $\S^{d-1}$ denoted 
 \begin{equation*}
   n'_u \mapsto n'_v  = {\rm Coupl}_{n_u,n_v}(n'_u) \in \S^{d-1},
 \end{equation*}
called \emph{spherical coupling}, satisfying $n'_u=n'_v$ if $n_u=n_v$, and equivalently defined as follows for $n_u \neq n_v$.
\begin{enumerate}[(i)]
\item $n'_v$ is obtained from $n'_u$ by performing the elementary rotation in $\Span(n_u,n_v)$ bringing  $n_u$ to $n_v$.
\item Denote by $t_u$ a tangent vector of $\S^{d-1}$ at base point $n_u$ of a geodesic of length $\theta$ bringing $n_u$ to $n'_u$. Generate $t_v$ from $t_u$ by using parallel transport in $\Span(n_u,n_v)$ from base point $n_u$ to base point $n_v$. Generate $n'_v$ as the endpoint of the geodesic of length $\theta$ and tangent to $t_v$ at base point $n_v$.
\end{enumerate}
Moreover, it satisfies by construction the symmetry condition
\begin{equation}
  \label{eq:sym}
   {\rm Coupl}_{n_v,n_u} =  {\rm Coupl}_{n_u,n_v}^{-1} .
\end{equation}
\end{Def}
It is necessary to keep in mind that the full mapping $(n_u,n_v) \mapsto  {\rm Coupl}_{n_u,n_v}$ is smooth, but has a singularity on the extremity set $\set{n_u,n_v \in \S^{d-1} \vert n_u=- n_v}$. This fact has already been pointed out (\cite{Tan78,FouMou09,FouMis13}) in slightly different contexts, and causes difficulty in order to define uniquely regular Levy generators and kinetic equations with such couplings. However, we will avoid such technical issues by using Grad's cut-off (see~\eqref{eq:grad}), and we will consider coupled Levy generators only on the formal level.

Anyway, the resulting probability transition is then a symmetric Markov coupling (by the symmetry condition~\eqref{eq:sym}).
\begin{Def}
A \emph{spherically coupled} random collision with deviation angle $\theta$ is defined by the following probability transition on coupled collisional directions (in $\S^{d-1} \times \S^{d-1}$):
\begin{align}
& c_{c,\theta}(n_u,n_v , \d n'_u \d n'_v )\eqdef \Big ( \one{n_u \neq -n_v} \delta_{ {\rm Coupl}_{n_u,n_v}(n'_u) }( \d n'_v)  \nonumber \\
& \hspace{1cm} +\one{n_u = -n_v} \delta_{ {\rm Coupl}_{n_u,\sigma}(n'_u) }( \d n'_v) {\rm Unif}_{\S^{d-1}}( \d \sigma ) \Big) c_{\theta}\pare{ n_u, \d n'_u }, \label{eq:c_coupl}
\end{align}
\end{Def}
\begin{Lem}
The spherically coupled probability transition~\eqref{eq:c_coupl} verifies the symmetry condition
\begin{align}\label{eq:sym2}
 c_{c,\theta}(n_u,n_v , \d n'_u \d n'_v ) =  c_{c,\theta}(n_v , n_u, \d n'_v \d n'_u ).
\end{align}
It is thus a symmetric Markov coupling of $c_\theta$.
\end{Lem}
\begin{proof}
  By construction, ${\rm Coupl}_{n_u,n_v}$ and ${\rm Coupl}_{n_u,\sigma}$ are isometries. On the other hand, by isotropy, for any isometry $R$ and vector $n_v \in \S^{d-1 }$ we have $R^{-1} c_\theta( R n_v, .)= c_\theta( n_v, .)$. Finally, the symmetry condition~\eqref{eq:sym} yields~\eqref{eq:sym2}.
\end{proof}

In Section~\ref{sec:contr}, the following key basic formula will be proved.
\begin{Lem}\label{lem:basic_contr}
  The Markov coupling defined by $c_{c,\theta}$ is almost surely contractive. Moreover, we have the average quadratic contraction estimate
\begin{equation}
    \label{eq:mdev5}
\int_{\S^{d-1} \times \S^{d-1}} \abs{n_u' - n_v'}^2\,  c_{c,\theta}(n_u,n_v , \d n'_u \d n'_v )  - \abs{n_u - n_v}^2 = -\sin^2 \theta \fracd{c_{d-1}}{c_{d-3}}\abs{n_u - n_v}^2 ,
  \end{equation}
where in the above $\dps c_{d} = \int_{0}^{\pi/2} \sin^{d}(\ph) \, \d \ph$ denotes the $d$'th Wallis integral.
\end{Lem}

We can now consider a coupled Nanbu (two-body) $N$-particle system (see \cite{Nan80}), each particle being endowed with \emph{two} velocities. It is a Markov process in state space $(\R^d \times \R^d)^N$, with generator of the form~\eqref{eq:full_gen_c}. The two-body Levy generator $L_c$ is a coupled collision operator defined when acting on two particles by:
\begin{equation}
  \label{eq:coupled_gen}
L_c(\psi)(u,v) \eqdef  \int_{ \S^{d-1} \times [0,\pi]}  \pare{ \psi(u',v') - \psi(u,v) } c_{c,\theta}(n_u,n_v, \d n'_u \d n'_v) \,  b(\d \theta ),
\end{equation}
for any test function $\psi$ on $(\R^d \times \R^d)^2$. The resulting particle system is by construction a formal Markov coupling of the usual Nanbu $N$-particle system process with two-body Levy generator $L$ given in~\eqref{eq:gen_levy}. We recall that such particle systems conserve total momentum and energy, and have the uniform probability on the sphere defined by these conservation laws as an invariant probability distribution.

In order to be rigorous (at least for uniqueness), let us recall that the latter process can be constructed using Grad's cut-off:
\begin{equation}
  \label{eq:grad}
  b_\eps(\d \theta) \eqdef \one{\theta \geq \eps} b( \d \theta ) \, \qquad \bar{b}_\eps \eqdef \int_{[0,\pi]} b_\eps( \d \theta) < +\infty .
\end{equation}
\begin{enumerate}[(i)]
\item Each particle perform a collision with a fixed rate $\bar{b}_\eps$, and with a uniformly randomly chosen other particle.
\item The scattering angle of each collision is independently sampled according to the probability $\fracd{ \dps b_\eps ( \d \theta) }{\dps \bar{b}_\eps}$.
\item The coupled random post-collisional directions $(n'_u,n'_v)$ (with scattering angle prescribed by (ii)) are sampled using the coupled isotropic probability transition on sphere $c_{c,\theta}$.
\end{enumerate}
The general case of Levy grazing collisions can then be considered as a the formal $\eps \to 0$ limit of the latter. 

\subsection{Kinetic equations}\label{sec:kinetic}
Consider a $N$-particle system. We say that propagation of chaos holds if the marginal distribution of $k$-particles ($k$ being fixed) is converging (in law) to a product measure when $N \to + \infty$. Under this assumption, the limit of the \emph{one body} distribution $\mu_t \in \calP(\R^d)$ of the particle system satisfies formally an evolution equation in closed form with a quadratic non-linearity given by:
\begin{equation}
  \label{eq:non-lin}
  \fracd{\d}{\d t} \int_{\R^d} \ph \, \d \mu_t = \int_{\R^d \times \R^d} L \pare{ \ph \otimes \one{} } \d \mu_t \otimes \mu_t,
\end{equation}
where in the above $\ph $ is a test function of $\R^d$. When $L$ is the collision operator~\eqref{eq:gen_levy} with given kernel $b$, then the non-linear equation~\eqref{eq:non-lin} is the Boltzmann equation in $\R^d$ with Maxwell collision kernel $b$. The usual expression on the particle velocity density, denoted $\dps f_t(v) \d v $, is then:
\begin{equation}
  \label{eq:boltz}
  \fracd{\d}{\d t} f_t(v) = \int_{\R^d \times \S^{d-1} \times [0,\pi]} \pare{f_t(v')f_t(v'_\ast) - f_t(v)f_t(v_\ast)} \d v_\ast \,  c_{\theta}(n_v , \d n'_{v_\ast}) \, b( \d \theta).
\end{equation}
In the above, the collision mapping~\eqref{eq:collision} is used implicitly, and detailed balance has been used to remove test functions. 

In the same way for the case coupled case, one obtains an evolution equation on the one body coupled particle distribution in the form
\begin{equation*}
  \label{eq:non-lin_coupled}
  \fracd{\d}{\d t} \int_{\R^d \times \R^d} \psi \mu_{c,t}( \d u \d v) = \int_{(\R^d \times \R^d)^2 } L_c \pare{ \psi \otimes \one{} } \mu_{c,t}(\d u \d v) \mu_{c,t}(\d u_\ast \d v_\ast).
\end{equation*}
When the underlying generator is $L_c$, the coupled collision operator~\eqref{eq:coupled_gen}, the associated non-linear kinetic equation in measure form is then
\begin{align}
&  \frac{\d}{\d t} \int_{\R^{2d}} \psi \, \d \mu_{c,t} = \int_{(\R^{d} \times \R^{d})^2} \int_{[0,\pi] \times \S^{d-1} \times \S^{d-1} } \nonumber \\[4pt]
& \hspace{1cm } \pare{\psi(u',v') - \psi(u,v)} c_{c,\theta}(n_u,n_v,\d n'_u \d n'_v) \, b(\d \theta) \mu_{c,t}(\d u \d v) \mu_{c,t}(\d u_\ast \d v_\ast).  \label{eq:non-lin_coupled_boltz}
\end{align}
wher $\mu_{c,t} \in \calP \pare{ \R^{d}\times \R^d}$ is the one body distribution of the coupled velocities. In the case of Boltzmann collisions, there is no hope to obtain on the coupled density (the density of $\mu_{c,t}$) an explicit simple expression similar to~\eqref{eq:boltz}. Indeed, this would require some kind of detailed balance in product space $\R^{d}\times \R^d$, which is broken for contractive coupling. 

Here again, the lack of smoothness of the mapping $(n_u,n_v) \mapsto c_{c,\theta}(n_u,n_v;\d n'_u,\d n'_v)$ at the set $\set{n_u,n_v \in \S^{d-1} \vert n_u=- n_v}$ is causing difficulties to obtain an appropriate Cauchy theory for~\eqref{eq:non-lin_coupled_boltz}. This difficulty is removed under Grad's cut-off~\eqref{eq:grad} (using, say, total variation distance).

\subsection{Results}\label{sec:results}
We can now detail the results of the present paper. 

We first compute the coupling creation functional~\eqref{eq:coupl_crea} of the coupled Nanbu particle system as defined by~\eqref{eq:cc_func}-\eqref{eq:coupled_gen}. The fact that the average of this functional can be bounded from above in some way by the coupling distance itself is not obvious. Fortunately, we have found a remarkable general inequality which enables to do so.

\begin{Pro}\label{pro:fund_ineq}
Let $(U,V) \in \R^d \times \R^d$ a couple of centered and normalized ($\E \abs{U}^2 =\E \abs{V}^2 =1$) random variables in euclidean space. Let $(U_\ast,V_\ast)\in \R^d \times \R^d$ be an i.i.d. copy. We have:
\begin{align}\label{eq:fund_ineq}
& f\pare{ \E \abs{U-V}^2} \leq  \min\pare{ \kappa_{\E\pare{U\otimes U}},  \kappa_{\E\pare{V\otimes V}} }   \nonumber \\
& \hspace{1cm} \times \E \pare{ \abs{U-U_\ast}^2\abs{V-V_\ast}^2 - \pare{\pare{U-U_\ast} \cdot \pare{V-V_\ast}}^2},
\end{align}
where in the above we have used the notation defined in~\eqref{eq:f}-\eqref{eq:kappa} (\textit{i.e.} $f(x) =f(4-x)=x-x^2/4$, and $\kappa_{S}=\pare{1- \normop{S}}^{-1} \in [d/(d-1),+\infty]0$ is a condition number of a symmetric positive matrix $S$ of trace $1$ and maximal eigenvalue $\normop{S}$). Moreover, a sufficient condition for the equality case in~\eqref{eq:fund_ineq} is given by the following isotropy and co-linear coupling conditions
\begin{enumerate}[(i)]
\item $
\frac{U}{\abs{U}} = \frac{V}{\abs{V}} \quad \as.
$
\item Either $\E\pare{U\otimes V} = \E\pare{U\otimes U}=\frac{1}{d} \Id$ or $\E\pare{U\otimes V} = \E\pare{V\otimes V}=\frac{1}{d} \Id$. 
\end{enumerate}
\end{Pro}
 
It is now possible to apply the inequality~\eqref{eq:fund_ineq} to coupled particle velocities $(u,v) \in (\R^d\times \R^d)^N$ with null momentum $\bracket{u}_{N} = \bracket{v}_{N} =0$ and normalized energy $ \bracket{\abs{u}^2}_{N} = \bracket{\abs{v}^2}_{N} =1$ on the probability space $([1,N],{\rm Unif})$ generated by averaging over particles. We first define the following $p$-moments:
\begin{itemize}
\item For any $\pN{x} \in \pare{\R^d}^N$,
\begin{equation*}
  \label{eq:m\pN{x}}
 m_{\pN{x},p} \eqdef  \bracket{ \abs{\pN{x}-\pN{x}_\ast}^{p} }_N^{1/p}.
\end{equation*}
\item For any random $X \in \R^d$, with $X_\ast \in \R^d$ an i.i.d. copy
\begin{equation*}
  \label{eq:mX}
m_{\Law(X),p} \eqdef  \E\pare{ \abs{X-X_\ast}^{p} }^{1/p}.
\end{equation*}
\item For any random $\pN{X} \in  \pare{\R^d}^N$,
\begin{equation*}
  \label{eq:m\pN{X}}
  m_{\Law(\pN{X}),p} \eqdef \E\pare{\bracket{ \abs{\pN{X}-\pN{X}_\ast}^{p} }_N}^{1/p}.
\end{equation*}
\item For any $p_0 \geq 1$ and any random $\pN{X} \in  \pare{\R^d}^N$,
\begin{equation*}
  \label{eq:mtilde\pN{X}}
  \widetilde{m}_{\Law(\pN{X}),p_0,p} \eqdef   \pare{\frac{d-1}{d}}^{p_0/p} \E\pare{ \kappa^{p_0}_{ \bracket{\pN{X} \otimes \pN{X}}_N} \bracket{ \abs{\pN{X}-\pN{X}_\ast}^{p} }_N}^{1/p}.
\end{equation*}
\end{itemize}
Using Hölder inequality, we obtain the following quasi-contractive estimate:
\begin{Pro}\label{pro:result}
  Let $(\pN{u},\pN{v}) \in (\R^d\times \R^d)^N$ satisfy the centering and normalization condition (conservation laws)~\eqref{eq:cons}. Let $\alpha >0 $ and $p_1,p_2 >1$ with $1/p_1 + 1/p_2 =1$ be given. Remark that
  \begin{equation}
    \label{eq:holder_sum}
    \frac{\alpha}{1+\alpha}+\frac{2+\alpha}{1+\alpha} \frac{1}{p_1(2+\alpha)} + \frac{2+\alpha}{1+\alpha} \frac{1}{p_2(2+\alpha)} =1.
  \end{equation}
Then, we have the inequality:
 \begin{align*}\label{eq:ineq_particle}
&\fracd{ f \pare{ \bracket{ \abs{\pN{u}-\pN{v}}^2  }_N } }{\bracket{\abs{\pN{u} -\pN{v} } \abs{\pN{u}_\ast -\pN{v}_\ast} - (\pN{u} -\pN{v}) \cdot (\pN{u}_\ast -\pN{v}_\ast ) }_N^{\frac{\alpha}{1+\alpha} } }\\
& \hspace{1cm}\leq  k_\alpha \min\pare{\kappa_{ \bracket{\pN{u} \otimes \pN{u}}_N}  ,\kappa_{ \bracket{\pN{v} \otimes \pN{v}}_N}  } m_{\pN{u},p_1(2+\alpha)}^{\frac{2+\alpha}{1+\alpha}}  m_{\pN{v},p_2(2+\alpha)}^{\frac{2+\alpha}{1+\alpha}},
  \end{align*}
where in the above we have used the bounded above and below constant
\begin{equation*}\label{eq:kalpha}
  k_\alpha \eqdef (1+\alpha)\pare{\fracd{2}{2+\alpha}}^{\frac{2+\alpha}{1+\alpha}} \quad (\xrightarrow[\alpha \to 0\, \text{or}\, +\infty]{} 2 \, \text{or} \,1).
\end{equation*}
\end{Pro}
Proposition~\eqref{pro:result} is the main result of this paper, and immediately shows that the trend to equilibrium of a Nanbu $N$-particle system is controlled by the velocity moments of order $2+\alpha >2$. It is useful to state the counterpart of Proposition~\ref{pro:result} for the formal limit $N=+\infty$.
\begin{Pro}\label{pro:result2}
  Let $(U,V) \in \R^d\times \R^d$ be two centered and normalized random vectors, and assume the isotropy condition:
\[
\E(U\otimes U) = \fracd{1}{d} \Id.
\]
Let $(U_\ast,V_\ast)$ be an i.i.d. copy. Denote $f(x) = x-x^2/4$. Then we have the inequality:
\begin{align*}
 \fracd{ f\pare{\E \abs{U-V}^2}  } {\dps \pare{ \E \bracket{\abs{U-U_\ast}\abs{V-V_\ast} - \pare{U-U_\ast} \cdot \pare{V-V_\ast}   }_N}^{\frac{\alpha}{1+\alpha} } }\leq k_\alpha \fracd{d}{d-1} m_{\Law(U),p_1(2+\alpha)}^{\frac{2+\alpha}{1+\alpha}}  m_{\Law(V),p_2(2+\alpha)}^{\frac{2+\alpha}{1+\alpha}}.
  \end{align*}
\end{Pro}
We then obtain consequences on the contraction in Wassertsein distance.
\begin{Cor}\label{cor:1}
  Any Nanbu particle system with Maxwell collisions $t \mapsto \pN{V}_t \in \pare{\R^d}^N$ with centering and normalization condition~\eqref{eq:cons} satisfies the following weak coupling - coupling creation inequality (for any $\alpha >0, 1/p_1+1/p_2 =1$) 
 \begin{equation*}\label{eq:ineq_particle_opt}
\fracd{\frac{1}{2} d^{2}_{\calW_2^N}(\Law(\pN{V}_t), {\rm Unif_{0,1}^N} ) } { \pare{- \fracd{1}{\lambda} \fracd{c_{d-3}}{c_{d-1} }\fracd{\d^+}{\d t}d^{2}_{\calW_2^N}(\Law(\pN{V}_t), {\rm Unif_{0,1}^N} ) }^{\frac{\alpha}{1+\alpha} }}\leq \frac{d}{d-1} k_\alpha \widetilde{m}_{{\rm Unif}^N_{0,1},p_1(1+\alpha),p_1(2+\alpha) }^{\frac{2+\alpha}{1+\alpha}}  m_{\Law(\pN{V}_t),p_2(2+\alpha)}^{\frac{2+\alpha}{1+\alpha}},
  \end{equation*}
The modified moment satisfies for any $p_0 \geq 1, p \geq 1$, $\widetilde{m}_{{\rm Unif}^N_{0,1},p_0,p} < +\infty$ as soon as $p_0 < (d-1)( (N-1)d+1)/2$, as well as
\begin{equation*}
  \label{eq:wish_conv}
  \lim_{N\to + \infty} \widetilde{m}_{{\rm Unif}^N_{0,1},p_0,p} = \E \pare{ \abs{N-N_\ast}^{p }}^{1/p} < +\infty .
\end{equation*}
where $(N,N_\ast)\in \R^d \times \R^d \sim \calN_{0,1} \times \calN_{0,1}$ is a couple of i.i.d. centered and normalized ($\E \abs{N}^2 =1$) normal random variables.
\end{Cor}
A similar result holds for the associated kinetic equation
\begin{Cor}\label{cor:2}
  The centered and normalized measure solution $(\mu_t)_{t\geq 0} \in \calP\pare{\R^d}$ of the space homogenous kinetic equation with Maxwell collisions~\eqref{eq:boltz} satisfies the following weak coupling - coupling creation inequality. Denote $\calN_{0,1}$ the associate reduced normal (Maxwellian) distribution in $\R^d$. Then we have
 \begin{equation*}\label{eq:ineq_particle_kin}
\fracd{\dps  f\pare{ d^{2}_{\calW_2}(\mu_t, \calN_{0,1} )  }}{\dps \pare{- \fracd{1}{\lambda} \fracd{c_{d-3}}{c_{d-1} }\fracd{\d^{+}}{\d t} d^{2}_{\calW_2}(\mu_t, \calN_{0,1} )  }^{\frac{\alpha}{1+\alpha} } }\leq \fracd{d}{d-1} k_\alpha m_{\calN\pare{0,1} ,p_1(2+\alpha)}^{\frac{2+\alpha}{1+\alpha}}  m_{\mu_t,p_2(2+\alpha)}^{\frac{2+\alpha}{1+\alpha}}.
  \end{equation*}
\end{Cor}

\subsection{Two counter-examples}\label{sec:cex}
Let us finally present two negative results, that demonstrates, in coupling- coupling creation inequalities, the necessity of sub-exponential power law estimates on the one hand, and the necessity to resort on higher moments of velocity distribution on the other hand. We give the counter-examples in the form lemmas, with proofs. In both cases, we consider a coupled distribution in the form of random variables
\begin{equation}
  \label{eq:as_cex}
  (U,V)\in \R^{d}\times \R^{d}, \quad U \sim \calN(0,\fracd{1}{d} \Id).
\end{equation}
which are centered and normalized, with normal distribution in the first variable. $(U_{\ast},V_{\ast})$ is an i.i.d. copy.

\begin{Lem}[The necessity of $>2$-moments]
Let~\eqref{eq:as_cex} holds. Denote the order $<2$ moment
\begin{equation}
  \label{eq:seq_1}
 m_q \eqdef \E   \pare{   \abs{ V}^{q} }^{1/q},
  \end{equation}
for some $0<q< 1$. Then we have the following degeneracy of the coupling creation
\[
\lim_{m_q \to 0}  {\E\pare{ \abs{U-U_{\ast}}\abs{V-V_{\ast}} - \pare{U-U_{\ast}} \cdot \pare{V-V_{\ast}}}}= 0.
\]
\end{Lem}
\begin{proof}
  Condition~\eqref{eq:seq_1} obviously implies via Hölder inequality that $\lim_{m_q \to 0} \E   \pare{   \abs{ V-U }^2 }= 2 \neq 0$. Moreover, we have
    \begin{align*}
&      \E \pare{ \abs{U -U_{\ast}}\abs{V-V_{\ast}} - \pare{U-U_{\ast}} \cdot \pare{V-V_{\ast}} } \leq 2  \E \pare{ \abs{U-U_{\ast}}\abs{V-V_{\ast}} }  \\
& \hspace{1cm} \leq   \E^{1/p}\pare{ \abs{U-U_{\ast}}^{p}} \E^{1/q}\pare{ \abs{V-V_{\ast}}^q } \xrightarrow[m_q \to +\infty]{} 0.
    \end{align*}
\end{proof}

\begin{Lem}[The necessity of sub-exponential rates]
 Let~\eqref{eq:as_cex} holds. Consider the co-linear coupling
\[
\fracd{V}{\abs{V}} \eqdef \fracd{U}{\abs{U} } \quad \as, 
\]
with moreover the following radial coupling perturbation on some interval $0 < r_- < r_+ < +\infty $:
\[
\abs{V} \eqdef \abs{U} \one{\abs{U} \notin [r_-,r_+]} + \E^{1/2}\pare{\abs{U}^2 \cond{} \abs{U} \in [r_-,r_+]}  \one{\abs{U} \in [r_-,r_+]}.
\]
Then we have the following degeneracy of the coupling - coupling creation estimate
\[
\lim_{r_- \to + \infty}\lim_{r_+ \to r_-} \fracd{f\pare{\E   \pare{   \abs{ V-U }^2}}}{\E \pare{ \abs{U-U_{\ast}}\abs{V-V_{\ast}} - \pare{U-U_{\ast}} \cdot \pare{V-V_{\ast}}}}= + \infty.
\]
\end{Lem}
\begin{proof}
  First, for such isotropic ($U$ is normally distributed) and co-linear couplings, the key inequality~\eqref{eq:fund_ineq} is in fact an equality. Denoting:
\[
A \eqdef \fracd{ \pare{U-U_\ast} \cdot \pare{V-V_\ast} }{\abs{U-U_\ast}\abs{V-V_\ast}},
\]
we obtain 
\[
\calR(r_-,r_+) \eqdef \fracd{f\pare{\E   \pare{   \abs{ V-U }^2}}}{\E \pare{ \abs{U-U_{\ast}}\abs{V-V_{\ast}} - \pare{U-U_{\ast}} \cdot \pare{V-V_{\ast}}}} = \fracd{d}{d-1} \fracd{\E\pare{\abs{U-U_\ast}^2 \abs{V-V_\ast}^2 \, (1-A^2) }} {\E\pare{\abs{U-U_\ast} \abs{V-V_\ast} \, (1-A)}}.
\]
Since $A=1$ when both $\abs{U} \notin [r_-,r_+]$ and $\abs{U_\ast} \notin [r_-,r_+] $, we have
\[
 (1-A) \leq  (1-A^2) \pare{ \one{\abs{U} \in [r_-,r_+]  } + \one{\abs{U_\ast} \in [r_-,r_+]  }} \, \as, \qquad 2 (1-A^2) \geq  (1-A^2) \pare{ \one{\abs{U} \in [r_-,r_+]  } + \one{\abs{U_\ast} \in [r_-,r_+]  }} \, \as,
\]
and the smoothness of Gaussian density yields
\[
\lim_{r_+ \to r_-} \calR(r_-,r_+) \geq  \fracd{d}{2(d-1)} \fracd{\E\pare{\abs{R_U-U_\ast}^2 \abs{R_U-U_\ast}^2} }{\E\pare{\abs{R_U-U_\ast} \abs{R_U-U_\ast}}},
\]
where $R_U$ is distributed uniformly on the sphere with radius $r_-$. Taking the limit $r_- \to + \infty$ yields the result.
\end{proof}

\section{Standard notation and facts}\label{sec:notation}
\subsection{Scattering}
As usual, the velocities of a pair of collisional particles are denoted
\[(v,v_{\ast}) \in \R^d\times \R^d,\] 
and the post-collisional quantities are denoted by adding the superscrpit~$'$. All particles are assumed to have the same mass so that the conservation of momentum imposes
\begin{equation*}
  v'+v'_\ast = v+v_\ast ,
\end{equation*}
and conservation of energy
\begin{equation*}
   \abs{v'}^2+\abs{v_\ast'}^2=\abs{v}^2+\abs{v_\ast}^2.
\end{equation*}
As a consequence, the relative speed is also conserved
\begin{equation*}
   \abs{v'-v_\ast'}=\abs{v-v_\ast},
\end{equation*}
and a collision can be fully described by only using the normalized velocity difference, called the \emph{collisional direction}. It is denoted
\begin{equation*}
  \label{eq:n}
  n_v \eqdef \fracd{v-v_\ast}{\abs{v-v_\ast}} \in \S^{d-1},
\end{equation*}
and in the same way $n'_v$ is called the \emph{post-collisional direction}. The collisional and post-collisional directions being given, velocities are then determined by the standard involutive collision mapping~\eqref{eq:collision} with inverse
\begin{equation*}
\label{eq:collision_inv}
    \begin{cases}
      v = \fracd{1}{2}(v'+v'_\ast) + \fracd{1}{2} \abs{v'-v'_\ast} n_v, \\
      v_\ast = \fracd{1}{2}(v'+v'_\ast) - \fracd{1}{2} \abs{v'-v'_\ast} n_v.
    \end{cases}
  \end{equation*}
The \emph{scattering or deviation angle} \[ \theta \in [0,\pi] \] of the collision is then uniquely defined as the half-line angle between the collisional and the post-collisional direction:
\begin{equation*}
  \label{eq:devangle}
  \cos \theta \eqdef \fracd{v'-v_\ast'}{\abs{v'-v_\ast'}} \cdot \fracd{v-v_\ast}{\abs{v-v_\ast}}  = n'_v \cdot n_v.
\end{equation*}

\subsection{Isotropic random walk on sphere}
Isotropy (or invariance by rotation, which is equivalent to the physical Galilean invariance in the case of velocity differences), does not impose any condition on the scattering angle $\theta$; however, it requires other degrees of freedom to be uniformly distributed. This is made precise in the following definition.
\begin{Def}
  Let $\theta \in [0,\pi]$ be given. The \emph{isotropic probability transition} on the sphere $\S^{d-1}$ with scattering angle $\theta \in [0,\pi]$ is the unique probability transition $
c_\theta: \S^{d-1} \to \calP(\S^{d-1})
$ invariant under isometries, and generating states at a prescribed angular distance $\theta$. Formally
\begin{equation}
  \label{eq:rand_rot}
  c_{\theta}(n_v, \d n'_v) \eqdef {\rm Unif}_{\set{n'_v \in \S^{d-1} \cond{} n_v \cdot n'_v = \cos \theta }}\pare{ \d n'_v },
\end{equation}
or equivalently
 \begin{equation}
  \label{eq:rand_rot_bis}
  c_{\theta}(n_v, . ) \mathop{=}^{{\Law}} {\rm rot}\pare{\theta,n_v,\Sigma_1,\Sigma_2}  n_v \quad \Sigma_1 , \Sigma_2 \sim {\rm Unif}_{\S^{d-1}} \otimes {\rm Unif}_{\S^{d-1}}, 
\end{equation}
where in~\eqref{eq:rand_rot_bis},  ${\rm rot}\pare{\theta,n_v,\sigma_1,\sigma_2}$ is the elementary\footnote{{\it i.e.} fixing the orthogonal of a plane} rotation of $\R^d$ with rotation angle $\theta$, rotation plane $\dps \Pi_{n_v,\sigma_1} \eqdef \Span(n_v,\sigma_1)$, and orientation prescribed by  $P_{\Pi_{n_v,\sigma_1} }(\sigma_1,\sigma_2)$ where $P$ stands for orthogonal projection.
\end{Def}
Clearly, invariance under isometries implies that the uniform distribution is an invariant distribution of $c_\theta$ . The expression~\eqref{eq:rand_rot_bis} is in fact useful to see why detailed balance (reversibility) hold; using the geometric inversion formula
\begin{equation*}
  \label{eq:rev_root}
  n'_v = {\rm rot}\pare{\theta,n_v,\sigma_1,\sigma_2} n_v \Leftrightarrow n_v = {\rm rot}\pare{\theta,n_v,-\sigma_1,-\sigma_2} n'_v ,
\end{equation*}
and the invariance of the uniform distribution under the parity transformation $\sigma \mapsto - \sigma$. This yields:
\begin{Lem}
  $c_\theta$ is reversible with invariant probability the uniform distribution. Formally,
\begin{equation*}
  \label{eq:rev}
\d n_v c_\theta ( n_v , \d n'_v ) = \d n_v '  c_\theta ( n'_v , \d n_v ).
\end{equation*}
\end{Lem}

\subsection{Two-body random collisions}
The following definition of a random collision will be useful (its properties are directly inherited from those of the isotropic probability transition on the sphere):
\begin{Def}
  Consider two collisional particles with velocity pair in euclidean $\R^d \times \R^d$. We call a \emph{(two body) random collision} with scattering angle $\theta$, the probability transition in $\R^{d}\times \R^d$ induced by the collision mapping~\eqref{eq:collision} and the isotropic probability transition on sphere $c_\theta(n_v, \d n'_v)$.  It satisfies:
\begin{enumerate}[(i)]
  \item Invariance (in law) under isometries of $\R^d$ (applied simultaneously to each particle velocity).
  \item Almost surely conservation of  the total energy and momentum (conservation laws).
  \item Reversibility with respect to uniform distributions. Formally,
    \begin{equation*}
      \label{eq:rev_bis}
      \d v \d v_\ast c_\theta ( n_v , \d n'_v ) = \d v ' \d v'_\ast  c_\theta ( n'_v , \d n_v ),
    \end{equation*}
    where in the above the collision mappings~\eqref{eq:collision}-\eqref{eq:collision_inv} are implicitly used.
  \end{enumerate}
\end{Def}

In the same way, we can consider continuous time collision (Levy) processes on $\R^{d}\times \R^d$, generated by a ( non-negative ) Levy measure:
\[
 b: \R^{+}_{\ast} \to \calM_{+}([0,\pi]),
\]
usually called an \emph{angular kernel}. Recall that Levy measures are measures with finite diffusive intensity~$\lambda$ as defined in~\eqref{eq:Levy}. The latter generates a Markov (isotropic Levy) process on the sphere $\S^{d-1}$ through the generator:
 \begin{equation*}
   \int_{[0,\pi]} \pare{\ph(n'_v) - \ph(n_v)}  c_{\theta}(n_v, \d n'_v) \, b(\d \theta).
 \end{equation*}
where in the above $\ph$ is a test function on $\S^{d-1}$. By extension using the collision mapping, we can define a \emph{two-body collision process} on velocity pairs ($\R^d \times \R^d$) with generator~\eqref{eq:gen_levy}. 
\begin{Lem}
The two-body collision Levy process with generator~\eqref{eq:gen_levy} satisfies the following properties:
  \begin{enumerate}[(i)]
  \item Invariant (in law) under isometries of $\R^d$ (applied simultaneously to velocity particles),
  \item Almost surely conservation of the total energy and momentum (conservation laws),
  \item Reversibility with respect to the uniform distribution.
    \end{enumerate}
\end{Lem}
\begin{Rem}
  In all the definitions above, Galilean invariance (rotation and translation for velocities) allow that $b$ may depend on the system state $(v,v_\ast)$ through the absolute collision speed $\abs{v-v_\ast}$ (a conserved quantity). Note that the fact that $\abs{v-v_\ast}$ is conserved is necessary to keep the reversibility properties. The case where $b$ is in fact independent of $\abs{v-v_\ast}$ is exactly what is called a \emph{Maxwell collision}.
\end{Rem}

\subsection{Two-body interacting particle systems}

When $L$ is the collision process~\eqref{eq:gen_levy}, the $N$-particle system generated by~\eqref{eq:full_gen} is called the Nanbu particle system.
\begin{Def}\label{def:Nanbu}
  The particle system in $\pare{\R^{d}}^N$ generated by the generator~\eqref{eq:full_gen}, with Maxwell collision generator~\eqref{eq:gen_levy} is called a \emph{Nanbu particle system} (\cite{Nan80}).
\end{Def}
The tensorial structure of~\eqref{eq:full_gen} implies:
\begin{Lem}
  If $L$ is stationary (or reversible) with state space $E$ and product invariant probability $\mu \otimes \mu \in E \times E$, then so is $\calL^N$ with product invariant probability $\mu^{ \otimes N } \in E^N$.
\end{Lem}
As a consequence, the Nanbu particle system satisfies the following properties, directly inherited from the two body case.
\begin{Lem}
A Nanbu particle system satisfies
\begin{enumerate}[(i)]
\item Invariance under isometries of $\R^d$ (applied to each particle velocity).
\item Almost surely conservation of total kinetic energy and momentum (conservation laws).
\item Reversibility (detailed balance) with uniform invariant distribution.
\end{enumerate}
\end{Lem}

\section{Coupled collisions}
In this section we will consider a pair of collisional particles with coupled velocities
\[ (u,v,u_\ast,v_\ast) \in \R^{ 2d } \times \R^{ 2d} . \] A coupled collision can then be described by expressing the post-collisional velocities
\[ (u',u'_\ast,v',v'_\ast) \in \R^{ 2d} \times \R^{ 2d }  \]
using coupled collision parameters. According to Section~\ref{sec:notation}, it is sufficient in order to obtain the above coupling to express, using the \emph{same} collision random parameters, the collision and post-collisional directions $(n_u, n'_u,n_v,n'_v) \in \pare{\S^{d-1}}^2 \times \pare{\S^{d-1}}^2$. This is \emph{naturally} done by what we have called in the present work a \emph{spherical coupling}. 

\subsection{Spherical coupling}\label{sec:coupl_2}
We give a special description of the isotropic probability transition with scattering angle $\theta$.
\begin{Lem}\label{lem:azim}
  Let $\theta \in [0,\pi]$ be given, as well as $ (n_v,m_v)$ two orthonormal vectors in $\S^{d-1}$. Consider the spherical change of variable
\begin{equation}
  \label{eq:azim}
  n'_v= \cos \theta \, n_v + \sin \theta \cos \ph \, m_v + \sin \theta \sin \ph \,  l \in \S^{d-1}
\end{equation}
where $\ph \in [0,\pi]$ is an \emph{azimuthal angle} and $l \in \S^{d-1}$ is such that $(n_v,m_v,l)$ is an orthonormal triplet. Then the image by the transformation~\eqref{eq:azim} of the probability distribution
\begin{equation}
  \label{eq:spher_vol}
  \sin^{d-3} \ph \, \fracd{\d \ph}{ c_{d-3}} {\rm Unif}_{(n_v,m_v)^{\perp} \cap  \, \S^{d-1} }(\d l),
\end{equation}
is the isotropic probability transition $c_\theta(n_v, \d n'_v)$ with initial state $n_v$ and scattering angle $\theta$ ($c_{d-3}$ denotes the Wallis integral normalization). In particular, the latter does not depend on the choice of $m_v$.
\end{Lem}
\begin{proof}
$c_\theta(n_v, \d n_v')$ is defined as the uniform distribution induced by the euclidean structure on the submanifold of $\S^{d-1}$ defined by $n'_v \cdot n_v =\cos \theta$. Moreover the expression of volume elements in (hyper)spherical coordinates implies that for any $m_v \in \S^{d-1}$, the vector $\cos \ph \, m_v + \sin \ph \,  l \in \S^{d-1} $ is distributed (under~\eqref{eq:spher_vol}) uniformly in the $d-2$-dimensional sphere $n_v^{\perp} \cap  \S^{d-1}$. The result follows.
\end{proof}
Of course in the above, only the scattering angle $\theta$ has an intrinsic physical meaning, the azimuthal angle $\ph$ being dependent of the arbitrary choice of the pair $(m_v,l)$. This leads to the core analysis of a spherical coupling.
\begin{Lem}\label{def:coupl}
  Let $(n_u,n_v) \in \S^{d-1} \times \S^{d-1}$ be given. A pair $(n'_u,n'_v) \in \S^{d-1} \times \S^{d-1}$ is spherically coupled (the spherical coupling mapping is defined in Definition~\ref{def:coupling}), in the sense that $n'_v={\rm Coupl}_{n_u,n_v}(n'_u)$ if $n_u\neq n_v$ and $n'_v={\rm Coupl}_{n_u,\sigma}(n'_u)$ for some $\sigma \in \S^{d-1}$ otherwise, if and only if
\begin{equation}
  \label{eq:sph_coupl}
  \begin{cases}
    n'_u= \cos \theta \, n_u + \sin \theta \cos \ph \, m_u + \sin \theta \sin \ph \,  l ,\\
    n'_v= \cos \theta \, n_v + \sin \theta \cos \ph \, m_v + \sin \theta \sin \ph \, l,
  \end{cases}
\end{equation}
where in the above $(n_u,m_u,l)$ and $(n_v,m_v,l)$ are both orthonormal sets of vectors such that $(n_u,m_u)$ and $(n_v,m_v)$ belong to the same plane have the same orientation with respect to $l$. Note that if $n_u\neq n_v$, the pair $(m_u,m_v)$ and the angle $\ph$ are defined uniquely up to a common involution (a change of sign of the vectors and the reflexion $\ph \to \pi - \ph$).
\end{Lem}
\begin{proof}
Assume $n_u\neq n_v$.  Denote by $R_\theta$ the unique elementary rotation bringing $n_u$ to $n_v$. By construction $R_\theta m_u = m_v$, and $R_\theta l = l$
\end{proof}
This immediately implies that the coupled probability transition $c_{c,\theta}(n_u,n_v,\d n_u \d n_v)$ is the image using the mapping~\eqref{eq:sph_coupl} above of the uniform probability described in $(\ph,l)$-variables in~\eqref{eq:spher_vol}.

\subsection{Contractivity of spherical couplings}\label{sec:contr}
We can now state the contractivity (''coupling creation'') equation satisfied by spherical couplings.
\begin{Lem} Let $(n_u,n_v) \in \S^{d-1} \times \S^{d-1}$ be given, with a spherically coupled pair $(n'_u,n'_v) \in \S^{d-1} \times \S^{d-1}$ (Definition~\ref{def:coupling}). Then the contractivity formula~\eqref{eq:mdev5} holds.
\end{Lem}
\begin{proof}
  Next, we expand $n_u' . n_v'$ using~\eqref{eq:sph_coupl} and obtain:
\begin{equation*}
n_u' . n_v' =  \pare{\cos \theta \, n_u + \sin \theta \cos \ph \, m_u}.\pare{\cos \theta \, n_v + \sin \theta \cos \ph \, m_v }+ \sin^2 \theta \sin^2 \ph.
  \end{equation*}
Next by construction, $(m_u,m_v)$ is obtained from a $\fracd{\pi}{2}$-rotation of $(n_u,n_v)$, so that $n_u . n_v = m_u . m_v$ and  $n_u.m_v=-m_u.n_v$ and
\begin{equation*}
n_u' . n_v' =  \pare{ \cos^2 \theta + \sin^2 \theta \cos^2 \ph } \, n_u.n_v  + \sin^2 \theta \sin^2 \ph .
  \end{equation*}
Using $ 1=  \cos^2 \theta + \sin^2 \theta \cos^2 \ph  + \sin^2 \theta \sin^2 \ph$ we obtain
\begin{equation*}
 n_u' . n_v' - n_u.n_v =  -  \sin^2 \theta \sin^2 \ph \pare{  n_u.n_v  -1 },
  \end{equation*}
and the result follows.
\end{proof}
We can then compute the consequence on velocities.
\begin{Lem}\label{lem:coupl_coll}
Consider coupled collisional and post-collisional velocities $(u,u_\ast,v,v_\ast) \in \R^{2d}\times \R^{2d}$, and a spherical coupling as defined by Definition~\ref{def:coupl}. Then we have:
\begin{align}
& \abs{u'-v'}^2 +\abs{u'_\ast-v'_\ast}^2-\abs{u-v}^2-\abs{u_\ast-v_\ast}^2 = \nonumber\\
& \hspace{1cm} - \sin^2\theta \sin^2\ph 
\pare{\abs{u-u_\ast}\abs{v-v_\ast} - (u-u_\ast) . (v-v_\ast) } \leq 0 . \label{eq:m2var}
\end{align} 
The above quantity vanishes if and only if the coupled collision directions are aligned with the same orientation ($n_u \cdot n_v =1$). If moreover the coupled velocities $(u,u_\ast)$ and $(v,v_\ast)$ have the same total momentum and energy, then we have the contractivity equality:
\begin{equation}\label{eq:contr_2}
\abs{u'-v'}^2 +\abs{u'_\ast-v'_\ast}^2-\abs{u-v}^2-\abs{u_\ast-v_\ast}^2 = - \sin^2\theta \sin^2\ph \fracd{1}{4}\pare{\abs{u-v}^2+\abs{u_\ast-v_\ast}^2 }.
\end{equation}
\end{Lem}
\begin{proof}
We use the following change of variable:
\begin{equation*}
  \begin{cases}
    \dps  s_v \eqdef \fracd{1}{2}\pare{v+v_\ast} \\[2pt]
    \dps d_v \eqdef \fracd{1}{2}\pare{v-v_\ast}
  \end{cases}
\Leftrightarrow
\begin{cases}
    \dps  v  = s_v + d_v \\[2pt]
    \dps v _\ast = s_v - d_v
  \end{cases}.
\end{equation*}
First remark that
\begin{align}
  \abs{u-v}^2  + \abs{u_\ast-v_\ast}^2 & = \abs{s_u - s_v + d_u -d_v}^2 + \abs{s_u - s_v - d_u + d_v}^2 \nonumber \\
& = 2 \abs{s_u - s_v}^2 + 2 \abs{d_u - d_v}^2 \label{eq:c2}
\end{align}
Developing the left hand side of~\eqref{eq:m2var}, and using the conservation laws ($s'=s$ and $\abs{d'}=\abs{d}$), we obtain
\begin{align*}
&  \abs{u'-v'}^2 +\abs{u'_\ast-v'_\ast}^2-\abs{u-v}^2-\abs{u_\ast-v_\ast}^2 \nonumber \\
&\quad = 2 \abs{d_u' -d_v'}^2 - 2 \abs{d_u -d_v}^2 \nonumber \\
& \quad =  - (u'-u'_\ast) \cdot (v '- v'_\ast) + (u-u_\ast) \cdot (v - v_\ast) \nonumber \\
& \quad = - \abs{u-u_\ast}\abs{v - v_\ast} \pare{n'_u \cdot n'_v - n_u \cdot n_v }\\
& \quad = \fracd{1}{2} \abs{u-u_\ast}\abs{v - v_\ast} \pare{ \abs{n'_u - n'_v}^2 - \abs{n_u - n_v}^2 }. 
\end{align*}
Then the contractivity formula~\eqref{eq:mdev5} yields the first result. If $(u,u_\ast)$ and $(v,v_\ast)$ have the same momentum and energy, then~\eqref{eq:c2} implies
\begin{align*}
  \abs{u-v}^2  + \abs{u_\ast-v_\ast}^2 & = 0 + 2 \abs{d_u}^2 + 2 \abs{d_v}^2  - 4 d_u \cdot d_v  \\
& = 4 \abs{u - u_\ast} \abs{v-v_\ast } \pare{ 1- n_u \cdot n_v   } , \\
\end{align*}
\end{proof}

\subsection{Coupled Nanbu particle systems}\label{sec:coupl_nanbu}
In this section, we detail properties of the coupled Nanbu particle system.

\begin{Def}
  A particle system in $\pare{\R^{2d}\times \R^{2d}}^N$ with generator~\eqref{eq:full_gen} associated to the two-body \emph{coupled} Levy collision generator $L_c$ in~\eqref{eq:coupled_gen} is called a \emph{(spherically) coupled Nanbu particle system}.
\end{Def}
Note that the possibility of constructing coupled Nanbu system relies on the use of Maxwell collisions. We indeed need a process where collision parameters are independent on the specific invariants (energy and momentum) of a particle pair. Again, uniqueness of such processes requires some further analysis, unless Grad's cut-off~\eqref{eq:grad} is used.
\begin{Lem}
  Consider a Nanbu particle system with two-body generator defined in ~\eqref{eq:coupled_gen}. Then, the resulting process is a coupling of the Nanbu particle system of Definition~\ref{def:Nanbu}.
\end{Lem}
\begin{proof}
(i) The fact that the marginal processes have the good distribution is a consequence of~\eqref{eq:sym2}. (ii) The fact that an identity coupling ($\pN{U}_0=\pN{V}_0 \as$) remains as such ($\pN{U}_t=\pN{V}_t \as \quad \forall t$)  is implied by the decrease of spherical couplings.
\end{proof}

By construction of the spherical coupling, a generated coupled Nanbu system then satisfies the following expected properties.
\begin{Lem}
  A coupled Nanbu particle system denoted $t \mapsto (\pN{U}_t,\pN{V}_t,\pN{U}_{\ast,t},\pN{V}_{\ast,t}) \in \pare{\R^{d}\times \R^{d}}^N$ satisfy the following properties.
  \begin{enumerate}[(i)]
  \item It is invariant under isometries of $\R^d$ (applied simultaneously to each of the $2\times N$ velocities).
\item The coupling is almost surely decreasing (for all $0\leq t \leq t +h$):
\[
\bracket{ \abs{\pN{U}_{t+h}-\pN{V}_{t+h}}^2 }_N  \leq \bracket{ \abs{\pN{U}_{t}-\pN{V}_{t}}^2}_N \quad \as .
\]
  \end{enumerate}
\end{Lem}
\begin{Rem}
  In the special case $N=2$, and if the coupled velocities have the same total momentum and kinetic energy, the system is exponentially contractive (from~\eqref{eq:contr_2}) with explicit rate:
\[
\fracd{\d}{ \d t} \ln \E\pare{\abs{\pN{U}_{t}-\pN{V}_{t}}^2 + \abs{\pN{U}_{\ast,t}-\pN{V}_{\ast,t}}^2}=- \fracd{1}{4} \lambda \fracd{c_{d-1}}{c_{d-3}}<0.
\]
  It is of interest to remark that the exponential rate is in fact uniform in the dimension, since:
\[
\lim_{d \to + \infty} \fracd{c_{d-1}}{c_{d-3}} = 1.
\]
\end{Rem}

We can finally  compute the coupling creation functional for the coupled Nanbu particle system.
\begin{Pro}
Consider the coupled Nanbu particle system in state space $\pare{\R^{2d}}^{N}$. Then the coupling creation functional is given by~\eqref{eq:coupl_crea}.
\end{Pro}
\begin{proof}
The calculation uses the definition of the Nanbu particle generator, the definition of the coupling creation functional $\cc$, and the two-body spherical coupling contractivity in Lemma~\eqref{lem:coupl_coll}. We have
\begin{align*}
&   \cc(u,v,u_\ast,v_\ast)  =  L_c \pare{(u,v,u_\ast,v_\ast) \mapsto \abs{u-v}^2+\abs{u_\ast-v_\ast}^2}    \\
& \quad = - \int_{[0,\pi]^2} \sin^2\theta \sin^2\ph  \, b(\d \theta)\,  \fracd{\sin^{d-3} \ph\d \ph }{c_{d-3}} \pare{ \abs{u-u_\ast}\abs{v-v_\ast} - (u-u_\ast) \cdot (v-v_\ast) }.
  \end{align*}
\end{proof}

\subsection{Proof of Corollary~\ref{cor:1} and~\ref{cor:2} }\label{sec:proofs}
The two corollaries~\ref{cor:1} and~\ref{cor:2} can now be deduced from Proposition~\ref{pro:result} and~\ref{pro:result2} respectively.
\begin{proof}[Proof of Corollary~\ref{cor:1}]
First, we need to extedn the moment inequality in Proposition~\ref{pro:result}. Taking the expectation, and using three-terms Hölder inequality on the expectation (as opposed to the particle averaging) with the power law decomposition~\eqref{eq:holder_sum} yields
\begin{align}\label{eq:m_av}
&  \E  f \pare{ \bracket{ \abs{\pN{U}-\pN{V}}^2  }_N  } \nonumber \\
& \quad \leq k_\alpha \E^{1/p_1(1+\alpha)}  \pare{  \kappa_{\bracket{\pN{U} \otimes \pN{U}_\ast}_N} ^{p_1(1+\alpha)} \bracket{\abs{\pN{U}-\pN{U}_\ast}^{p_1(2+\alpha)}}_{N}} \nonumber \\
& \qquad \times \E^{1/p_2(1+\alpha)}\bracket{\abs{\pN{V}-\pN{V}_\ast}^{p_2(2+\alpha)}}_{N} \nonumber \\
& \qquad \times \E^{\alpha/(\alpha+1)}\bracket{\cc\pare{\pN{U},\pN{U}_\ast,\pN{V},\pN{V}_\ast}}_N.
\end{align}
We can the remark that for any pair $(\mu,\nu) \in \calP(\R^d)\times \calP(\R^d)$ satisfying the conservations laws, $d^2_{\calW^2}(\mu,\nu) \leq 2 $ (try the trivial product coupling). As a consequence, we obtain the general upper bound for random exchangeable particle systems satisfying the conservation laws:
\[
d^2_{\calW_2} \pare{\eta_{\pN{V}},\eta_{\pN{U}}} \leq 2  f\pare{ \bracket{ \abs{\pN{V}-\pN{U} }^2 }_N} \quad \as,
\]
so that by definition of Wasserstein optimal coupling
\[
d^2_{\calW^N_2} \pare{\Law(\pN{V}),\Law(\pN{U})} \leq 2  \E f\pare{ \bracket{ \abs{\pN{V}-\pN{U} }^2 }_N }.
\]

Let now $t_0 \geq 0$ be given, and consider as an initial condition $(\pN{\opt{U}}_{t_0},\pN{\opt{V}}_{t_0}) \in (\R^d \times \R^d)^N$, a random variable representation of an $d_{\calW_2}$-optimal coupling between ${\rm Unif}^N_{0,1}$ and $\Law(\pN{V}_{t})$. Consider next the solution $h \mapsto (\pN{U}_{\eps,t+h},\pN{V}_{\eps,t+h})$ of the coupled Nanbu collision process with the latter initial condition, and subject to Grad's cut-off~\eqref{eq:grad} with $\eps>0$. By construction, the latter satisfies
  \begin{align*}
 &   d^2_{\calW^N_2} \pare{\Law(\pN{V}_{\eps,t+h}),{\rm Unif}^N_{0,1}} - d^2_{\calW^N_2} \pare{\Law(\pN{V}_{\eps,t}),{\rm Unif}^N_{0,1}} \\
& \hspace{3cm} \leq - \int_{t}^{t+h} \E \bracket{\cc\pare{ \pN{U}_{\eps,t+h'} , \pN{V}_{\eps,t+h'} , \pN{U}_{\ast,\eps,t+h'} , \pN{V}_{\ast,\eps,t+h'}  }}_N \, \d h'.
  \end{align*}
We can then combine the inequality~\eqref{eq:m_av}, with the continuity with respect to weak convergence of $\eps \to \Law(\pN{V}_{\eps,t+h})$ (this is a standard result of approximation of Levy processes by jump processes with bounded intensity, see~\cite{EthKur85}). Since the Wassertsein distance metrizes weak convergence and moments are continuous bounded observables, the result follows.

Finally, the integrability condition and the limit in~\eqref{eq:wish_conv} can be justified using the explicit expression for the probability distribution of eigenvalues of Wishart matrices (see \textit{e.g}~\cite{And58,SchKriChat73}). To prove the limit, consider some $M >0$ and higher finite $l$-moments of the following random variable
\[
\kappa_{\bracket{\pN{U} \otimes \pN{U} }_N} \one{\kappa_{\bracket{\pN{U} \otimes \pN{U} }_N \leq M}}.
\] 
My Markov inequality, the result will follow from dominated convergence and the fact that $\kappa_{\bracket{\pN{U} \otimes \pN{U} }_N} $ has finite $l$-moments uniformly bounded in $N$ for any $l\geq 1$. Using the explicit expression of Wishart ensemble distribution (see Section~$1$ in~\cite{SchKriChat73}), such moments can be bounded above by ($\Gamma$ is the usual so-called function)
\[
\fracd{\Gamma(Nd/2)}{\Gamma((N-l)d/2)}\fracd{\prod_{i=0}^{d-1}\Gamma((N-i-l)/2)}{\prod_{i=0}^{d-1}\Gamma((N-i)/2)} \limop{\sim}_{N \to + \infty}  N^{ld/2} \times (N^{-l/2})^d \to 1.
\]
The result follows.
\end{proof}
\begin{proof}[Proof of Corollary~\ref{cor:2}]
The proof is similar to the one of Corollary~\ref{cor:1}, except that we need to justify the continuity of $\eps \to \mu_{\eps,t}$ with respect to weak convergence, as well as uniform control of higher moments, where $\mu_{\eps,t}$ the solution of the kinetic equation with Grad's cut-off. The continuity can found in Section~$5$ of~\cite{TosVil99}. The uniform control on higher moments is standard for Maxwell molecules, where explicit computations of the latter can be carried out (see the classical paper~\cite{IkeTru56}).
\end{proof}

\section{Coupling / coupling creation inequalities}
In this section, the key inequalities between coupling and coupling creation are proven. The key point is to compare the \emph{coupling $L^2$-distance}, with the \emph{average parallelogram square area} spanned by the difference of two independent copies. This is the content of the first section.

\subsection{The special inequality}\label{sec:spec}
Let $(U,V)$ be two random vectors in a Euclidean space,  and $(U_\ast,V_\ast)$ an i.i.d. copy. We assume the latter are \emph{centered}
\begin{equation}
  \label{eq:cent}
  \E(U)=\E(V)=0.
\end{equation}
If $(Z_1,Z_2)$ are two centered random vectors, we will use the notation
\begin{equation*}
  C_{Z_1,Z_2} \eqdef \E \pare{Z_1 \otimes Z_2} .
\end{equation*}
The goal is to bound from above the \emph{coupling distance} $\E\pare{ \abs{U-V}^2 }$ with the following alignement average
\begin{equation*}
  \E\pare{ \abs{U-U_\ast}^2\abs{V-V_\ast}^2 - \pare{ \pare{U-U_\ast} \cdot \pare{V-V_\ast}}^2  },
\end{equation*}
which is the average parallelogram area spanned by the two vector differences $U-U_\ast$ and $V-V_\ast$.

The main computation is based on an expansion of the paralleogram area formula
and a general trace inequality for positive definite symmetric matrices in even dimension. In the present context, the latter matrix is given by the full co-variance operator
\begin{equation*}
  C_{(U,V),(U,V)} = \bmat C_{U,U} & C_{U,V} \\ 
C_{V,U} & C_{V,V} \emat,
\end{equation*}
and the trace inequality is detailed in the following lemma.
\begin{Lem}\label{lem:trace}
  Let $U$ and $V$ be two centered random vectors in $\R^d$. Then we have 
\[
\Tr\pare{C_{U,U} C_{V,V} } - \Tr \pare{C_{U,V} C_{V,U} }  \leq \min\pare{ \fracd{\normop{C_{U,U}}}{\Tr \pare{ C_{U,U} } },\fracd{\normop{C_{V,V}}}{\Tr \pare{ C_{V,V} }}} \pare{\Tr\pare{C_{U,U} } \Tr\pare{ C_{V,V} } - \Tr \pare{C_{U,V} }^2}, 
\]
where $\normop{.}$ denotes the spectral radius (maximal eigenvalue) of a symmetric non-negative operator. Moreover, the equality case holds if (sufficient condition) either $C_{U,U}$ and $C_{U,V}$ or $C_{V,V}$ and $C_{U,V}$ are co-linear to the identity matrix.
\end{Lem}
\begin{proof}
First assume that $C_{U,U}$ has only strictly positive eigenvalues. In an orthonormal basis where $C_{U,U}$ is diagonal, we have the expression, for $i,j,k \in [\! [1,d]\!]$:
\begin{align*}
  \Tr\pare{C_{U,U} C_{V,V}-C_{U,V} C_{V,U} }    & =   \sum_{i}  C^{i,i}_{U,U} C^{i,i}_{V,V} -  \sum_{j,k}   C^{j,k}_{U,V} C^{k,j}_{V,U} ,\\
& = \sum_{i} \pare{ C^{i,i}_{U,U} C^{i,i}_{V,V} - \pare{ C^{i,i}_{U,V} }^2} \underbrace{-\sum_{j \neq k}  \pare{C^{j,k}_{U,V}}^2}_{\leq 0}.
\end{align*}
Then, by definition of the maximal eigenvalue of $C_{U,U}$:
\begin{align*}
  \sum_{i} C^{i,i}_{U,U} C^{i,i}_{V,V} - \pare{ C^{i,i}_{U,V} }^2 & \leq \normop{C_{U,U}} \pare{\Tr\pare{C_{V,V}} - \sum_{i} \fracd{\pare{ C^{i,i}_{U,V} }^2 }{C^{i,i}_{U,U} } }, \\
\end{align*}
so that by Cauchy-Schwarz inequality
\begin{align*}
 \pare{\sum_{i}  C^{i,i}_{U,V} }^2 \leq \pare{ \sum_{i} \fracd{\pare{ C^{i,i}_{U,V} }^2 }{C_{U,U}^{i,i} } } \times  \pare{ \sum_{i} C_{U,U}^{i,i}}, 
\end{align*}
and we eventually get
\begin{align*}
  \Tr\pare{C_{U,U} C_{V,V}-C_{U,V} C_{V,U} }    \leq \fracd{\normop{C_{U,U}}}{\Tr\pare{C_{U,U}}} \pare{\Tr\pare{C_{V,V}} \Tr\pare{C_{U,U}} - \Tr \pare{C_{U,V} }^2}.
\end{align*}
The general case of degenerate eigenvalues is obtained by density.
\end{proof}
The expansion of the average square paralellogram area is detailed in the next lemma.
\begin{Lem}\label{lem:eq_al}
Let $U$ and $V$ be two centered random vector in $\R^d$. Let $(U_\ast,V_\ast)$ be a i.i.d. copy. 
Then we have the following decomposition:
\begin{align}\label{eq:eq_al}
 & \E\pare{ \abs{U-U_\ast}^2\abs{V-V_\ast}^2 - \pare{ \pare{U-U_\ast} \cdot \pare{V-V_\ast}}^2  } \nonumber \\
& \quad  = \underbrace{\E\pare{ \abs{U}^2\abs{V}^2 - \pare{U \cdot V}^2 }}_{\geq 0} +  \underbrace{\Tr\pare{ \pare{C_{U,V} -C_{V,U}}\pare{C_{V,U} -C_{U,V}}}}_{\geq 0} \nonumber\\
& \qquad + 2\underbrace{ \pare{  \Tr\pare{C_{U,U} } \Tr\pare{ C_{V,V} } - \Tr \pare{C_{U,V} }^2 - \Tr\pare{C_{U,U} C_{V,V} } + \Tr \pare{C_{U,V} C_{V,U} } }}_{ \dps \mathop{\geq}^{(\text{Lemma~\ref{lem:trace}})} \pare{1-\min\pare{ \frac{\normop{C_{U,U}}}{\Tr \pare{ C_{U,U} } } ,\frac{\normop{C_{V,V}}}{\Tr \pare{ C_{V,V} }}} } \pare{\Tr\pare{C_{U,U} } \Tr\pare{ C_{V,V} } - \Tr \pare{C_{U,V} }^2}} .
\end{align}
\end{Lem}
\begin{proof} Before computing terms, recall that if $M,N$ are two square matrices, then
\[
\Tr\pare{M N} = \Tr\pare{NM} =\Tr\pare{M^T N^T}= \Tr\pare{N^T M^T}.
\]
Let us expand the alignement functional (the left hand side of~\eqref{eq:eq_al}). We have first,
  \begin{align*}
    \E\pare{ \abs{U-U_\ast}^2\abs{V-V_\ast}^2 } &= 2 \E\pare{\abs{U}^2\abs{V}^2} +2\E\pare{\abs{U}^2} \E \pare{\abs{V}^2}   + 4  \E\pare{U \cdot U_\ast \, V \cdot V_\ast} + 8 \times 0\\
&=  2 \E\pare{\abs{U}^2\abs{V}^2} + 2 \Tr\pare{C_{U,U}} \Tr\pare{C_{V,V}}  + 4 \Tr \pare{ C_{U,V} C_{V,U}},
  \end{align*}
and second,
\begin{align*}
   &\hspace{-1cm} \E\pare{\pare{ \pare{U-U_\ast} \cdot \pare{V-V_\ast}}^2} \\
=& 2 \E\pare{U \cdot V^2} +  2 \E\pare{U \cdot V}^2 +  2 \E \pare{U \cdot V_\ast^2} +  2 \E \pare{U_\ast \cdot  V \, U \cdot V_\ast} + 8 \times 0 \\
=& 2 \E\pare{U \cdot V^2} + 2 \Tr\pare{C_{U,V}}^{2}  + 2 \Tr \pare{C_{U,U}C_{V,V}} + 2 \Tr \pare{C_{U,V}^2}. 
  \end{align*}
On the other hand, 
\begin{align*}
    \Tr\pare{ \pare{C_{U,V} -C_{V,U}}\pare{C_{V,U} -C_{U,V}}} = -2 \Tr \pare{C^2_{U,V}} + 2\Tr \pare{C_{U,V}C_{V,U}},
  \end{align*}
and the result then follows.
\end{proof}
Two remarks.
\begin{Rem}
In the papers, the random vectors will be \emph{normalized}
\begin{equation}
  \label{eq:norm}
  \E(\abs{U}^2)=\E(\abs{V}^2)=1,
\end{equation}
so that we have
$
  \Tr \pare{ C_{U,U} } = \Tr \pare{C_{V,V} } = 1,
$
and
\begin{equation*}\label{eq:norm_coupl}
  \Tr\pare{C_{U,U} } \Tr\pare{ C_{V,V} } - \Tr \pare{C_{U,V} }^2 = \fracd{1}{2} f\pare{ \E\pare{\abs{U-V}^2} }, 
\end{equation*}
with the function
$
f(x) = x -\fracd{x^2}{4} \limop{\sim}_{x \to 0} x  
$
defined in~\eqref{eq:f}. Then Lemma~\ref{lem:trace} and~\ref{lem:eq_al} immediately yield Proposition~\ref{pro:fund_ineq}.
\end{Rem}
\begin{Rem}
  Assuming that the normalization~\eqref{eq:norm} above holds, and that $C_{U,U} = \fracd{1}{d} \Id$ is isotropic, then~\eqref{eq:eq_al} becomes
\begin{align}\label{eq:spe_2}
 &  f\pare{ \E\pare{\abs{U-V}^2} }  \leq \frac{d}{d-1} \E\pare{ \abs{U-U_\ast}^2\abs{V-V_\ast}^2 - \pare{ \pare{U-U_\ast} \cdot \pare{V-V_\ast}}^2  }
\end{align}
Moreover, a sufficient condition for equality in~\eqref{eq:spe_2} is given by strongly istropic distributions with co-linear coupling, defined by the fact that the lengths $(\abs{U},\abs{V}) \in \R_+^2$ are independant of the identically coupled and uniformly distributed direction $ \fracd{U}{\abs{U}} (= \fracd{V}{\abs{V}} \, \as)$.
\end{Rem}

\subsection{Proof of coupling creation estimates}\label{sec:holder}
It is now possible to use Hölder inequality to relate the coupling distance with the coupling contraction functional~\eqref{eq:coupl_crea}, by using the special inequality~\eqref{eq:fund_ineq}. This yields Proposition~\ref{pro:result} and Proposition~\ref{pro:result2}. Here are the proofs.
\begin{proof}[Proof of Proposition~\ref{pro:result} and~\ref{pro:result2}]
It is a consequence of Hölder inequality applied to the right hand side of~\eqref{eq:fund_ineq}. We first apply~\eqref{eq:fund_ineq} on the pair $(\pN{u},\pN{v})$ with respect to the probability space generated by the particle averaging operator $\bracket{ \quad }_N$. We obtain using the conservation laws~\eqref{eq:cons}:
\begin{align*}
 & I \eqdef \bracket{ \abs{\pN{u}-\pN{u}_\ast}^2\abs{\pN{v}-\pN{v}_\ast}^2 - \pare{ \pare{\pN{u}-\pN{u}_\ast} \cdot \pare{\pN{v}-\pN{v}_\ast}}^2  }_N \nonumber \\
 & \qquad \geq \pare{ 1-\normop{\bracket{\pN{u} \otimes \pN{u}_\ast}_N} }f(\bracket{ \abs{\pN{u}-\pN{v}}^2}_N).
\end{align*}
Next, let us denote 
\[
\al_{+/-} = \abs{\pN{u}-\pN{u}_\ast} \abs{\pN{v}-\pN{v}_\ast} +/-  \pare{\pN{u}-\pN{u}_\ast} \cdot \pare{\pN{v}-\pN{v}_\ast},
\]
and introduce $q = 1 + \alpha $, $p=1+1/\alpha$ so that $1/p+1/q = 1$. using Hölder inequality two times yields
\begin{align*}
&  I = \bracket{  \al_{+} \al_{-}  }_N  = \bracket{  \al_{+}\al_{-}^{1/q}  \al_{-}^{1/p}  }_N  \\
& \leq  q\pare{\fracd{2}{q+1}}^{1/q+1} \bracket{\abs{\pN{u}-\pN{u}_\ast}^{1+1/q}\abs{\pN{v}-\pN{v}_\ast}^{1+1/q}\al_{-}^{1/p}}_N \\
& \leq  q\pare{\fracd{2}{q+1}}^{1/q+1} \bracket{\abs{\pN{u}-\pN{u}_\ast}^{1+q}\abs{\pN{v}-\pN{v}_\ast}^{1+q}}_{N}^{1/q} \bracket{\al_{-}}_N^{1/p}\quad \as,
\end{align*}
where in the line before last line, we have use the sharp inequality $(1-\theta)(1+\theta)^{1/q}\leq q\pare{\fracd{2}{q+1}}^{1/q+1}$ that holds for any $\theta \in [-1,1]$. Then, remarking that $k_\alpha =q\pare{\fracd{2}{q+1}}^{1/q}$ and $p_1(1+q)/qp_1 = \fracd{\alpha+2}{\alpha+1}$ yields the result.

The case of Corollary~\ref{pro:result2} is similar with $\E$ formally replacing $\bracket{\, . \,}_{N}$.
\end{proof}

\bibliographystyle{plain}
\bibliography{boltzmann}
\end{document}